\documentclass[reqno]{amsart}

\usepackage{amsmath}
\usepackage{amsthm}
\usepackage{amssymb}
\usepackage{mathrsfs}
\usepackage[latin1]{inputenc}
\usepackage{diagrams}
\usepackage{a4wide}

\title{Ricci-flat deformations and special holonomy}
\author{Johannes Nordström}

\subjclass[2000]{53C25}
\address{Department of Mathematics, Imperial College London,
    London SW7 2AZ, UK}
\email{j.nordstrom@imperial.ac.uk}

\title{Ricci-flat deformations of metrics with exceptional holonomy}

\newcommand{\ignore}[1]{}

\DeclareMathOperator{\tr}{tr}

\newcommand{\ad}{\mathrm{ad}}

\newcommand{\half}{{\textstyle\frac{1}{2}}}

\newcommand{\bbr}{\mathbb{R}}
\newcommand{\bbc}{\mathbb{C}}

\newcommand{\bbrp}{\mathbb{R}^{+}}

\newcommand{\halg}{\mathfrak{h}}
\newcommand{\glalg}{\mathfrak{gl}}

\newcommand{\gstr}{$G$\nobreakdash-\hspace{0pt}structure}
\newcommand{\gtstr}{$G_{2}$\nobreakdash-\hspace{0pt}structure}
\newcommand{\spstr}{$Spin(7)$-structure}
\newcommand{\spnstr}{$Sp(n)$-structure}

\newcommand{\gmfd}{$G$\nobreakdash-\hspace{0pt}manifold}
\newcommand{\gtmfd}{$G_{2}$\nobreakdash-\hspace{0pt}manifold}
\newcommand{\spmfd}{$Spin(7)$\nobreakdash-\hspace{0pt}manifold}
\newcommand{\gmetric}{$G$\nobreakdash-\hspace{0pt}metric}
\newcommand{\gtmetric}{$G_{2}$\nobreakdash-\hspace{0pt}metric}
\newcommand{\spmetric}{$Spin(7)$\nobreakdash-\hspace{0pt}metric}
\newcommand{\ltwoorth}{$L^{2}$\nobreakdash-\hspace{0pt}orthogonal}

\newcommand{\sustr}{$SU(n)$\nobreakdash-\hspace{0pt}structure}

\newcommand{\holda}[1]{C^{#1, \alpha}}
\newcommand{\holdax}[2]{C^{#1, \alpha}_{#2}}
\newcommand{\inc}{e}
\newcommand{\harm}{\mathcal{H}}
\newcommand{\norm}[1]{\Vert #1 \Vert}

\newcommand{\slt}{K}
\newcommand{\inner}[1]{{<}#1{>}}
\newcommand{\linner}[1]{{<}#1{>}_{L^{2}}}

\newcommand{\calm}{\mathcal{M}}
\newcommand{\calw}{\mathcal{W}}
\newcommand{\calq}{\mathcal{Q}}
\newcommand{\cals}{\mathcal{S}}
\newcommand{\cali}{\mathcal{I}}

\newcommand{\cald}{\mathcal{D}}

\newcommand{\qcalz}{\calq}
\newcommand{\qcalg}{\calr}
\newcommand{\pcal}[1]{\mathcal{#1}}
\newcommand{\defstrp}{\defstrg}
\newcommand{\defmetp}{\defmetz}
\newcommand{\kpcal}[1]{\mathcal{#1}_{k+1}}
\newcommand{\defstr}{\calm}
\newcommand{\defmet}{\calw}
\newcommand{\defstrg}{\defstr_{G}}
\newcommand{\defmetg}{\defmet_{G}}

\newcommand{\defmetz}{\defmet_{0}}

\newcommand{\calr}{\mathcal{R}}

\newcommand{\lmgt}{\Lambda_{G_{2}}}
\newcommand{\lmsp}{\Lambda_{Spin(7)}}
\newcommand{\lmsu}{\Lambda_{SU(n)}}
\newcommand{\lmspn}{\Lambda_{Sp(n)}}
\newcommand{\gspace}{\Lambda_{G}T^{*}M}
\newcommand{\gtspace}{\Lambda_{G_{2}}T^{*}M}
\newcommand{\spspace}{\Lambda_{Spin(7)}T^{*}M}
\newcommand{\suspace}{\Lambda_{SU(n)}T^{*}M}
\newcommand{\spnspace}{\Lambda_{Sp(n)}T^{*}M}

\newcommand{\scrmkz}{\mathscr{M}^{k}_{Z}}
\newcommand{\holdad}[1]{C^{#1, \alpha}_{\delta}}

\newcommand{\caldkpz}{\mathcal{D}^{k+1}_{Z}}

\newcommand{\cptopenthm}{I}
\newcommand{\eacopenthm}{I$'$}

\newtheorem{thm}{Theorem}[section]
\newtheorem*{theoremi}{Theorem I}

\newtheorem*{theoremip}{Theorem I$'$}

\newtheorem{prop}[thm]{Proposition}
\newtheorem{lem}[thm]{Lemma}

\newtheorem{cor}[thm]{Corollary}
\theoremstyle{definition}
\newtheorem{defn}[thm]{Definition}
\theoremstyle{remark}
\newtheorem{rmk}[thm]{Remark}

\begin{document}

\begin{abstract}
Let $G$ be one of the Ricci-flat holonomy groups $SU(n)$, $Sp(n)$, $Spin(7)$ or
$G_2$, and $M$ a compact manifold of dimension $2n$, $4n$, $8$ or $7$,
respectively.  We prove that the natural map from the moduli space of
torsion-free \gstr s on $M$ to the moduli space of Ricci-flat metrics
is open, and that the image is a smooth manifold.
For the exceptional cases $G = Spin(7)$ and $G_2$ we extend the result
to asymptotically cylindrical manifolds.
\end{abstract}

\maketitle

\section{Introduction}

The possible holonomy groups of simply-connected non-symmetric irreducible
Riemannian manifolds were classified by Berger \cite{berger55}.
`Berger's list' contains several infinite families, and the two exceptional
cases $Spin(7)$ and $G_{2}$, appearing as the holonomy of manifolds of
dimension $8$ and $7$ respectively.
In many cases, an effective approach to studying \gmetric s
(by which we mean metrics with holonomy contained in $G$)
is to define them in terms of certain closed differential forms, equivalent to
\emph{torsion-free \gstr s}. A \gstr{} defines a Riemannian metric,
and if its torsion vanishes (which is a first-order differential equation)
then the induced metric has holonomy contained in $G$.
For $G = SU(n)$, $Sp(n)$, $Spin(7)$ or $G_{2}$ we define a
\emph{\gmfd{}} to be a connected oriented manifold of dimension $2n$, $4n$, $8$
or $7$ respectively, equipped with a torsion-free \gstr{} and the associated
Riemannian metric.

\gmetric s are Ricci-flat for $G = SU(n)$, $Sp(n)$, $Spin(7)$ or $G_{2}$.
For compact manifolds M.Y. Wang \cite[Theorem 3.1]{wang91} proved a local
converse: any small Ricci-flat deformation of a \gmetric{} still has holonomy
contained in $G$.
In other words, the moduli space $\defmetg$ of \gmetric s is an open subset of
the moduli space $\defmetz$ of Ricci-flat metrics.
This is an analogue of a result of Koiso \cite{koiso83} on deformations of
Kähler-Einstein metrics.
Wang proves the result case by case, but asks if there is a general proof.

In this paper we observe that the problem can be reduced in a uniform way to
showing unobstructedness for deformations of torsion-free \gstr s. This has in
turn been given a uniform treatment by Goto \cite{goto04}. As part of the
proof we provide a clear summary of the deformation theory of Ricci-flat
metrics on a compact manifold (a special case of deformation theory for
Einstein metrics used by Koiso \cite{koiso83}). This treatment makes it
easier to extend the results to other types of complete manifolds, and we will
discuss the asymptotically cylindrical case in some detail.

If $M$ is a compact \gmfd{}
then the group $\mathcal{D}$ of diffeo\-morphisms of $M$ isotopic to the
identity acts on the space of torsion-free \gstr s by pull-backs.
The resulting quotient is the moduli space $\defstrg$ of torsion-free \gstr s
on $M$, and is known to be a manifold. 
This is due to Tian \cite{tian87} and Todorov \cite{todorov89} in the
Calabi-Yau ($G = SU(n)$) case, and Joyce in the exceptional cases
(see \cite[\S $10.4$, $10.7$]{joyce00}).
$\mathcal{D}$ also acts on the space of Riemannian metrics, and we let
$\defmetg$ and $\defmetz$ denote the moduli spaces of \gmetric s and Ricci-flat
metrics respectively.
In \S \ref{ricdefgsec} we prove

\begin{theoremi}
Let $G = SU(n)$, $Sp(n)$, $Spin(7)$ or $G_{2}$, and let $M$ be a compact \gmfd.
Then $\defmetg$ is open in $\defmetz$.
Moreover, $\defmetg$ is a smooth manifold and the natural map
\begin{equation*}
m : \defstrg \to \defmetg
\end{equation*}
that sends a torsion-free \gstr{} to the metric it defines is a submersion.
\end{theoremi}

\begin{rmk}
It is easy to see that $\defmetg$ is also closed in $\defmetz$, so it is a
union of connected components.
It seems to be an open problem whether there exist \emph{any} compact Ricci-flat
manifolds without a holonomy reduction.
\end{rmk}

\begin{rmk}
The quotient of the space of \gmetric s by the group of \emph{all}
diffeomorphisms of $M$ (not just the ones isotopic to the identity)
is a quotient of~$\defmetg$ with discrete fibres and in general an orbifold
(cf. remark \ref{orbirmk}).
\end{rmk}

The case $G = G_{2}$ of theorem \cptopenthm{} was proved by
M.Y. Wang \cite[Theorem $3.1$B]{wang91}.
For $G = Sp(n)$ or $Spin(7)$, Wang showed that $\defmetg \subseteq \defmetz$ is
open (so the statement of theorem \cptopenthm{} is stronger).
Manifolds with holonomy in $SU(n)$ are Calabi-Yau manifolds, i.e. Ricci-flat
Kähler manifolds. The case $G = SU(n)$ of theorem \cptopenthm{} is therefore a
special case of a more general result by Koiso on Einstein deformations of
Kähler-Einstein metrics.

Let $X^{2n}$ be a compact Kähler-Einstein manifold. Koiso
\cite[Theorem $0.7$]{koiso83} shows that if the Einstein constant $e$
(equivalently the first Chern class $c_{1}(X)$) is non-positive and the complex
deformations of $X$ are unobstructed, then any small Einstein deformation of
the metric is Kähler with respect to some perturbed complex structure.
In other words, the map from the moduli space of Kähler-Einstein structures
to the moduli space of Einstein metrics is open
(see e.g. \cite[\S 12J]{besse87} for a discussion).
The proof shows that near any Kähler-Einstein metric there is a smooth
pre-moduli space of Einstein metrics, so that the moduli space of
Kähler-Einstein metrics is an orbifold.
Tian \cite{tian87} and Todorov \cite{todorov89} show that on a compact
Calabi-Yau manifold the obstructions to the complex deformations vanish.
Hence theorem \cptopenthm{} for $G = SU(n)$ follows from Koiso's theorem,
except for the claim that $\defmet_{SU(n)}$ is smooth (and not just an
orbifold).

\begin{rmk}
Dai, X. Wang and Wei \cite{dai05} use the fact that $\defmetg$ is open
in $\defmet$ to deduce that any scalar-flat deformation of a Ricci-flat
\gmetric{} on a compact manifold remains a \gmetric.
\end{rmk}

The proof of theorem \cptopenthm{} given in \S \ref{ricdefgsec} is a
simplification of Wang's argument for the case $G = G_{2}$. First, we observe
that the point-wise surjectivity of the derivative of $m$ follows from a
well-known property of Laplacians on manifolds with reduced holonomy noted by
Chern \cite{chern57}. This makes it easy to see that the proof applies also for
the other Ricci-flat holonomy groups, provided that the deformations of
torsion-free \gstr s are unobstructed.
Second, we streamline some parts of the deformation theory for Ricci-flat
metrics. This makes it easier to generalise the result to certain non-compact
settings.

One relevant type of complete non-compact Riemannian manifolds are
exponentially asymptotically cylindrical (EAC) ones
(defined in \S \ref{eacsub}).
In \S \ref{eacopensub} we explain that the necessary deformation theory for
Ricci-flat metrics and (at least when $G = Spin(7)$ or~$G_{2}$) torsion-free
\gstr s carries over to the EAC case, so that there are smooth moduli spaces
$\defstrg$ and $\defmetz$ of torsion-free EAC \gstr s and Ricci-flat EAC
metrics on an EAC \gmfd{} $M$.

\begin{theoremip}
Let $G = Spin(7)$ or $G_{2}$, and $M$ an EAC \gmfd.
Then $\defmetg$ is open in $\defmetz$.
Moreover, $\defmetg$ is a smooth manifold and the natural map
\begin{equation*}
m : \defstrg \to \defmetg
\end{equation*}
is a submersion.
\end{theoremip}

In \cite{kovalev05} Kovalev proves the analogous result for EAC Calabi-Yau
manifolds, by an extension of Koiso's arguments for the compact Kähler-Einstein
case.
The discussion in subsection \ref{eacopensub} of deformations of EAC Ricci-flat
metrics is similar to that in \cite{kovalev05}, while the necessary results on
deformations of EAC \gstr s are taken from \cite{jn1}.

\begin{rmk}
One may consider the structure of the map $m$ in greater detail.
For the exceptional cases $G = G_{2}$ and $Spin(7)$, one can use the
characterisation of torsion-free \gstr s in terms of parallel spinors
(cf. M.Y. Wang \cite{wang89}) to show that $m$ is a diffeomorphism if the
holonomy of $M$ is exactly $G$ for any \gmetric{} (this depends only on the
topology of $M$), and that $\defstrg$ is in general a disjoint union of fibre
bundles over $\defmetg$ with real projective plane fibres (the components
corrspond to different spin structures on $M$). See \cite[\S 5.3]{jnthesis} for
details.
For Calabi-Yau manifolds, \cite[Theorem $12.103$]{besse87} states that the
moduli space of Calabi-Yau structures on a compact manifold is a locally
trivial fibration with compact fibres over the moduli space of Calabi-Yau
metrics (but does not describe the fibres further).
\end{rmk}

\smallskip

\noindent\emph{Acknowledgements.}
I am grateful to Alexei Kovalev for many helpful discussions.

\section{Preliminaries}
\label{prelimsec}

We describe how a metric with holonomy $G = Spin(7)$, $G_{2}$, $SU(n)$
or $Sp(n)$ can be defined in terms of a torsion-free \gstr. This is a set of
differential forms, that are both parallel and harmonic.
For more background on manifolds with special holonomy see e.g.
Joyce \cite{joyce00} or Salamon \cite{salamon89}.

\subsection{Holonomy}
\label{holsub}

We define the holonomy group of a Riemannian manifold.
For a fuller discussion of holonomy see e.g. \cite[Chapter $2$]{joyce00}.

\begin{defn}
Let $M^{n}$ be a manifold with a Riemannian metric $g$.
If $x \in M$ and $\gamma$ is a closed piecewise $C^{1}$ loop in $M$ based at
$x$ then the parallel transport around $\gamma$ (with respect to the Levi-Civita
connection) defines an orthogonal linear map
$P_{\gamma} : T_{x}M \to T_{x}M$. The \emph{holonomy group}
$Hol(g, x) \subseteq O(T_{x}M)$ at $x$ is the group generated
by $\{P_{\gamma} : \gamma \textrm{ is a closed loop based at } x\}$.
\end{defn}

If $x, y \in M$ and $\tau$ is a path from $x$ to $y$ we can define a
group isomorphism $Hol(g, x) \to Hol(g, y)$ by 
$P_{\gamma} \mapsto P_{\tau} \circ P_{\gamma} \circ P^{-1}_{\tau}$.
Provided that $M$ is connected we can therefore identify $Hol(g, x)$ with
a subgroup of $O(n)$, independently of $x$ up to conjugacy,
and talk simply of the holonomy group of~$g$.

There is a correspondence between tensors fixed by the holonomy group and
parallel tensor fields on the manifold.

\begin{prop}[{\cite[Proposition $2.5.2$]{joyce00}}]
\label{simpleholprop}
Let $M^{n}$ be a Riemannian manifold, $x \in M$
and $E$ a vector bundle on $M$ associated to $TM$.
If $s$ is a parallel section of $E$ then $s(x)$ is preserved by
$Hol(g, x)$.
Conversely if $s_{0} \in E_{x}$ is preserved by $Hol(g, x)$ then there is a
parallel section $s$ of $E$ such that $s(x) = s_{0}$.
\end{prop}

\begin{defn}
Let $M^{n}$ a manifold and $G \subseteq O(n)$ a closed subgroup.
A \emph{\gmetric} on $M$ is a metric with holonomy contained in $G$.
\end{defn}

\subsection{$Spin(7)$-structures}
\label{sp7sub}

The stabiliser in $GL(\bbr^{8})$ of
\begin{multline}
\label{sp7formeq}
\psi_{0} = dx^{1234} + dx^{1256} + dx^{1278} + dx^{1357} - dx^{1368}
-dx^{1458} - dx^{1467} - \\ dx^{2358} - dx^{2367} - dx^{2457} + dx^{2468}
+ dx^{3456} + dx^{3478} + dx^{5678} \in \Lambda^{4}(\bbr^{8})^{*}
\end{multline}
is $Spin(7)$ (identified with a subgroup of $SO(8)$ by the spin
representation).
For an oriented vector space $V$ of dimension $8$ let
$\lmsp V^{*} \subset \Lambda^{4}V^{*}$ be the subset of forms equivalent to
$\psi_{0}$ under some oriented linear isomorphism
$V \cong \bbr^{8}$.
A \emph{\spstr} on an oriented manifold $M^{8}$ is a section of the subbundle
$\spspace \subset \Lambda^{4}T^{*}M$.
Since $Spin(7) \subset SO(8)$ a \spstr{} $\psi$ naturally
defines a Riemannian metric $g_{\psi}$ on $M$. Note that $\psi$ is self-dual
with respect to this metric.

We make a note of the decomposition of $\Lambda^{4}\bbr^{8}$ into irreducible
representations of $Spin(7)$. Firstly it splits into the self-dual and
anti-self-dual parts $\Lambda^{4}_{\pm}\bbr^{8}$. We let
$\Lambda^{4}_{d}\bbr^{8}$ denote an irreducible component of rank $d$.
Then
\begin{subequations}
\label{spindecompeq}
\begin{gather}
\Lambda^{4}_{+}\bbr^{8} = \Lambda^{4}_{1}\bbr^{8} \oplus
\Lambda^{4}_{7}\bbr^{8} \oplus \Lambda^{4}_{27}\bbr^{8} , \\
\Lambda^{4}_{-}\bbr^{8} = \Lambda^{4}_{35}\bbr^{8} .
\end{gather}
\end{subequations}

The tangent space at $\psi$ to the space of \spstr s $\Gamma(\spspace)$
is $\Gamma(E_{\psi})$, where $E_{\psi} \subset \Lambda^{4}T^{*}M$ is
a $Spin(7)$-invariant linear subbundle. More precisely, the \spstr{} $\psi$
determines a decomposition of $\Lambda^{4}T^{*}M$ modelled on
(\ref{spindecompeq}),
and $E_{\psi} = \Lambda^{4}_{1\oplus7\oplus35}T^{*}M$.

A \spstr{} $\psi$ is \emph{torsion-free}
if it is parallel with respect to the metric it induces.
It follows immediately from proposition \ref{simpleholprop}
that a metric $g$ on $M^{8}$ has holonomy contained in $Spin(7)$ if
and only if it is induced by a torsion-free \spstr s.

The condition that $Hol(g) \subseteq Spin(7)$ imposes algebraic constraints on
the curvature of $g$. In particular any \spmetric{} is Ricci-flat
(see \cite[Corollary $12.6$]{salamon89}).
The torsion-free condition for $\psi$ can usefully be rewritten as $d\psi = 0$
(see \cite[Lemma $12.4$]{salamon89}).

\subsection{$G_{2}$-structures}

Recall that $G_{2}$ can be defined as the automorphism group of the normed
algebra of octonions. Equivalently, $G_{2}$ is the stabiliser in
$GL(\bbr^{7})$ of
\begin{equation}
\label{g2formeq}
\varphi_{0} = dx^{123} + dx^{145} + dx^{167} + dx^{246}
 - dx^{257} - dx^{347} - dx^{356} \in \Lambda^{3}(\bbr^{7})^{*} .
\end{equation}
For an oriented vector space $V$ of dimension $7$ let
$\lmgt V^{*} \subset \Lambda^{3}V^{*}$ be the subset consisting of forms
equivalent to $\varphi_{0}$ under some oriented linear isomorphism
$V \cong \bbr^{7}$.
A \emph{\gtstr} on an oriented manifold $M^{7}$ is a section $\varphi$ of the
subbundle $\gtspace \subset \Lambda^{3}T^{*}M$, and naturally
defines a Riemannian metric $g_{\varphi}$ on~$M$.

The typical fibre of $\gtspace$ is isomorphic to $GL(\bbr^{7})/G_{2}$, so by
dimension-counting $\gtspace$ is an open subbundle of $\Lambda^{3}T^{*}M$.
Thus the tangent space at $\varphi$ to the space of \gtstr s $\Gamma(\gtspace)$
is $\Omega^{3}(M) = \Gamma(E_{\varphi})$, if we let
$E_{\varphi} = \Lambda^{3}T^{*}M$.

A \gtstr{} $\varphi$ is \emph{torsion-free}
if it is parallel with respect to the metric it induces.
A metric $g$ on $M^{7}$ has holonomy contained in $G_{2}$ if
and only if it is induced by a torsion-free \gtstr.

\gtmetric s are Ricci-flat (see \cite[Proposition $11.8$]{salamon89}).
As observed by Gray, a \gtstr{} $\varphi$ is torsion-free if and only if
$d\varphi = 0$ and $d^{*}_{\varphi}\varphi = 0$ (where the
codifferential $d^{*}_{\varphi}$ is defined using the metric induced by
$\varphi$, see \cite[Lemma $11.5$]{salamon89}).

\subsection{$SU(n)$-structures}

Let $z^k = x^{2k-1} +ix^{2k}$ be complex coordinates on $\bbr^{2n}$. Then 
the stabiliser in $GL(\bbr^{2n})$ of the pair of forms
\begin{subequations}
\label{su3formseq}
\begin{gather}
\label{sl3formeq}
\Omega_{0} = dz^1 \wedge \cdots \wedge dz^n \in
\Lambda^{n}(\bbr^{2n})^{*} \otimes \bbc \\
\label{kahlerformeq}
\omega_{0} = {\textstyle\frac{i}{2}}(dz^1 \wedge d\bar z^1 + \cdots
dz^n \wedge d\bar z^n) \in \Lambda^{2}(\bbr^{2n})^{*}
\end{gather}
\end{subequations}
is $SU(n)$.
For an oriented real vector space $V$ of dimension $2n$
let $\lmsu V^{*} \subset \Lambda^{n}V^*_\bbc \oplus \Lambda^{2}V^{*}$ be the
subset of pairs $(\Omega,\omega)$ equivalent to $(\Omega_{0},\omega_{0})$
under some oriented isomorphism $V \cong \bbr^{2n}$.
An \emph{\sustr} on an oriented manifold $M^{2n}$ is a section $(\Omega,\omega)$
of the subbundle
$\suspace \subset \Lambda^{n}T^{*}_\bbc M \oplus \Lambda^{2}T^{*}M$.
It naturally defines an almost complex structure
and a Riemannian metric on~$M$, such that $\Omega$ has type $(n,0)$.
The volume form on $V$ is given
by both $(-1)^{\frac{n(n-1)}{2}}(\frac{i}{2})^n \Omega \wedge \bar\Omega$
and $\frac{1}{n!} \omega^{n}$ (cf. Hitchin \cite[\S 2]{hitchin97}).

An \sustr{} is \emph{torsion-free} if it is parallel with respect to the
metric it induces, and a metric on $M^{2n}$ has holonomy contained in
$SU(n)$ if and only if it is induced by a torsion-free \sustr.

$(\Omega, \omega)$ is torsion-free if and only if
$d\Omega = d\omega = 0$. Then the induced almost complex structure
is integrable, the Riemannian metric is a Ricci-flat Kähler metric,
and $\Omega$ is a holomorphic $(n,0)$-form.
$M^{2n}$ equipped with a torsion-free \sustr{} is called an
\emph{$SU(n)$-manifold} or \emph{Calabi-Yau \mbox{$n$-fold}}.

\subsection{$Sp(n)$-structures}

Let $q^k = x^{4k-3} +ix^{4k-2} + jx^{4k-1} + kx^{4k}$ be quaternionic
coordinates on~$\bbr^{4n}$. Then we may write
\[ dq^1 \wedge d\bar q^1 + \cdots + dq^n \wedge d\bar q^n
= -2(i\omega^I_0 + j\omega^J_0 + k\omega^K_0) , \]
with $\omega^I_0, \omega^J_0, \omega^K_0 \in \Lambda^2(\bbr^{4n})^*$.
The stabiliser in $GL(\bbr^{4n})$ of this triple of $2$-forms is $Sp(n)$.
For an oriented real vector space $V$ of dimension $4n$ let
$\lmspn V^* \subset (\Lambda^2 V^*)^{\otimes 3}$ be the subset of triples
$(\omega^I, \omega^J, \omega^K)$ equivalent to
$(\omega^I_0, \omega^J_0, \omega^K_0)$ under some oriented isomorphism
$V \cong \bbr^{4n}$.
An \emph{\spnstr{}} on an oriented manifold $M^{4n}$ is a section
of the subbundle $\spnspace \subset (\Lambda^2 T^*M)^{\otimes 3}$. It is
\emph{torsion-free} if it is parallel with respect to the induced metric, and a
metric on $M^{4n}$ has holonomy contained in $Sp(n)$ if and only if it is
induced by a torsion-free \spnstr.

Equivalently, $(\omega^I, \omega^J, \omega^K)$ is torsion-free if and only if
$d\omega^I = d\omega^J = d\omega^K = 0$. Then the metric of $M$ is Ricci-flat,
and $M$ has a triple $I, J, K$ of anti-commuting integrable complex structures, such that $\omega^I$ is the Kähler form and $\omega^J + i \omega^K$ a
holomorphic $(2,0)$-form with respect to $I$, etc.
$M^{4n}$ equipped with a torsion-free \spnstr{} is called an
\emph{$Sp(n)$-manifold} or \emph{hyperKähler manifold}.

\subsection{Laplacians}
\label{harmholsub}

For a Riemannian manifold with holonomy $H$ one can define a
\emph{Lichnerowicz Laplacian} on vector bundles associated to
the $H$-structure.
On differential forms this agrees with the usual Hodge Laplacian, as is
explained in Besse \cite[\S 1I]{besse87}.
This can be used to define decompositions of the spaces of harmonic forms
analogous to the Kähler decomposition on a Kähler manifold,
as observed by Chern \cite{chern57}.

Suppose a Riemannian manifold $M^{n}$ has holonomy group $Hol(M) \subseteq H$
(where $H$ is a closed subgroup of $O(n)$),
and a corresponding $H$-structure. Let $\rho : H \to GL(E)$ be a representation
of $H$, and $E_{\rho}$ the corresponding associated vector bundle.
Let $\halg_{\ad}$ be the vector bundle induced by the adjoint representation.
$\halg_{\ad}$ can be identified with a subbundle of~$\Lambda^{2}T^{*}M$,
and because $Hol(M) \subseteq H$ the Riemannian curvature tensor $R$ is a
(symmetric) section of $\halg_{\ad} \otimes \halg_{\ad}$.
We use the Lie algebra representation $D\rho : \halg \to End(E)$
to define
\[ (D\rho)^{2} : \halg \otimes \halg \to End(E), \;\:
a \otimes b \mapsto D\rho(a) \circ D\rho(b) . \]
This induces a bundle map $\halg_{\ad} \otimes \halg_{\ad} \to End(E_{\rho})$.
The symmetry of $R$ implies that $(D\rho)^{2}(R)$ is a self-adjoint section
of $End(E_{\rho})$.

\begin{defn}
\label{lichdef}
Let $M$ be a Riemannian manifold with $Hol(M) \subseteq H$ and $\rho$
a representation of $H$.
The \emph{Lichnerowicz Laplacian} on the associated vector bundle $E_{\rho}$
is the elliptic formally self-adjoint operator
\[ \triangle_{\rho} = \nabla^{*}\nabla - 2(D\rho)^{2}(R) \; : \;
\Gamma(E_{\rho}) \to \Gamma(E_{\rho}) , \]
where $\nabla$ is the connection on $E_{\rho}$ induced by the
Levi-Civita connection on $M$.
\end{defn}

\begin{lem}
\label{lichinstlem}
Let $M^{n}$ be a Riemannian manifold.
The Lichnerowicz Laplacian corresponding to the standard representation
of $O(n)$ on $\Lambda^{m}(\bbr^{n})^{*}$ is the usual Hodge Laplacian
$\triangle$ on $\Lambda^{m}T^{*}M$.
\end{lem}

\begin{proof}
See \cite[Equation ($1.154$)]{besse87}.
\end{proof}

\begin{lem}[{cf. \cite[Theorem $3.5.3$]{joyce00}}]
\label{lichlem}
Let $M$ be a Riemannian manifold with $Hol(M) \subseteq H$ and
\mbox{$\phi : E \to F$} an equivariant map of $H$-representations
$(E, \rho)$, $(F, \sigma)$. $\phi$ induces a bundle map
$E_{\rho} \to F_{\sigma}$, and the diagram below commutes.

\begin{diagram}[PostScript=Rikicki,balance]
\Gamma(E_{\rho})	&  \rTo^{\phi}  &  \Gamma(F_{\sigma}) \\
\dTo^{\triangle_{\rho}} &  		&  \dTo_{\triangle_{\sigma}} \\
\Gamma(E_{\rho})	&  \rTo^{\phi}  &  \Gamma(F_{\sigma}) 
\end{diagram}
In particular, if $\rho_{1}, \rho_{2}$ are $H$-representations then
$\triangle_{\rho_{1} \oplus \rho_{2}} = \triangle_{\rho_{1}}
\oplus \triangle_{\rho_{2}}$.
\end{lem}

\begin{proof}
Clear from the fact that the Lichnerowicz Laplacian is defined
naturally by the representations.
\end{proof}

Suppose that $\Lambda^{m}(\bbr^{n})^{*}$ splits as a direct sum of
representations of $H$. On a manifold $M$ with holonomy contained in $H$
there is a corresponding splitting of $\Lambda^{m}T^{*}M$ into
$H$-invariant subbundles.
Lemma \ref{lichlem} implies that the Hodge Laplacian commutes with the
projections to the subbundles. Hence there is also a decomposition for
the harmonic forms (see \cite[Theorem $3.5.3$]{joyce00}).

\subsection{Asymptotically cylindrical manifolds}
\label{eacsub}

A non-compact manifold $M$ is said to have \emph{cylindrical ends}
if $M$ is written as union of two pieces $M_{0}$ and $M_{\infty}$ with
common boundary $X$, where $M_{0}$ is compact,
and $M_{\infty}$ is identified with $X \times \bbrp$
by a diffeomorphism (identifying $\partial M_{\infty}$ with $X \times \{0\}$).
$X$~is called the \emph{cross-section} of~$M$.
Let $t$ be a smooth real function on $M$ which is the $\bbrp$-coordinate
on $M_{\infty}$, and negative on the interior of $M_{0}$.
A tensor field $s$ on $M$ is said to be \emph{exponentially asymptotic}
with rate $\delta > 0$ to a translation-invariant tensor $s_{\infty}$ on $M$
if $e^{\delta t} \norm{\nabla^{k}(s-s_{\infty})}$ is bounded
on $M_{\infty}$ for all $k \geq 0$,
with respect to a norm defined by an arbitrary Riemannian metric on $X$.

A metric on $M$ is called \emph{EAC} (exponentially asymptotically cylindrical)
if it is exponentially asymptotic to a product metric on $X \times \bbr$.
Similarly a \gstr{} is said to be EAC if it is exponentially asymptotic to a
translation-invariant \gstr{} on $X \times \bbr$ which defines a product
metric.
A diffeomorphism $\phi$ of $M$ is called EAC if it is exponentially
close to a product diffeomorphism $(x,t) \mapsto (\Xi(x), t + h)$ of
$X \times \bbr$ in a similar exponential sense.

\begin{rmk}
If an EAC metric has reduced holonomy then so does the induced metric on the
cross-section. In particular, the cross-section of an EAC \spmfd{} is a
compact \gtmfd, and the cross-section of an EAC \gtmfd{} is a Calabi-Yau
$3$-fold.
\end{rmk}

On an asymptotically cylindrical manifold $M$ it is useful to introduce
\emph{weighted Hölder norms}. Let $E$ be a vector bundle on $M$ associated
to the tangent bundle, $k \geq 0$, $\alpha \in (0,1)$ and
$\delta \in \bbr$. We define the $\holdad{k}$-norm of a section $s$ of $E$
in terms of the usual Hölder norm by
\begin{equation}
\norm{s}_{\holdad{k}} = \norm{e^{\delta t}s}_{\holda{k}} .
\end{equation}
Denote the space of sections of $E$ with finite $\holdad{k}$-norm
by $\holdad{k}(E)$.
Up to Lipschitz equivalence the weighted norms are independent of the choice
of asymptotically cylindrical metric, and of the choice of $t$ on the compact
piece $M_{0}$. In particular, the topological vector spaces $\holdad{k}(E)$ are
independent of these choices.

The main importance of the weighted norms is that elliptic asymptotically
translation-invariant operators acts as Fredholm operators on the weighted
spaces of sections.
In particular, this applies to the Hodge Laplacian of an EAC metric.

\begin{thm}
\label{lapindthm}
Let $M$ be an asymptotically cylindrical manifold.
If $\delta > 0$ with $\delta^{2}$ smaller than any positive eigenvalue
of the Laplacian on $X$ then
\begin{equation}
\label{lapmapeq}
\triangle : \holdax{k+2}{\pm\delta}(\Lambda^{m}T^{*}M) \to
\holdax{k}{\pm\delta}(\Lambda^{m}T^{*}M)
\end{equation}
is Fredholm for all $m$.
The index of \textup{(\ref{lapmapeq})} is $\mp (b^{m-1}(X) + b^{m}(X))$.
\end{thm}

\begin{proof}
The Fredholm result is a special case of Lockhart and McOwen
\cite[Theorem 6.2]{lockhart85},
while the index formula can be found in Lockhart \cite[\S 3]{lockhart87} (or
Melrose \cite[\S 6.4]{melrose94}).
\end{proof}

This can be used to deduce results analogous to Hodge theory for compact
manifolds. Let $\harm^m_0$ denote the space of bounded harmonic $m$-forms on
$M$, and $\harm^m_\infty$ the translation-invariant harmonic forms on
$X \times \bbr$. $\harm^m_\infty = \harm^m_X \oplus dt \wedge \harm^{m-1}_X$,
where $\harm^m_X$ are the harmonic forms on $X$. Any $\phi \in \harm^m_0$ is
asymptotically translation-invariant; let $B(\phi) \in \harm^m_\infty$ denote
its limit. We can write $B(\phi) = B_a(\phi) + dt \wedge B_e(\phi) \in
\harm^m_X \oplus dt \wedge \harm^{m-1}_X$. Then
\[ \harm^m_0 = \harm^m_{abs} \oplus \harm^m_E , \]
where $\harm^m_{abs}$ is the kernel of $B_e : \harm^m_0 \to \harm^{m-1}_X$,
and $\harm^m_E \subset \harm^m_0$ is the subspace of exact forms.

\begin{thm}
\label{absderhamthm}
Let $M$ be an EAC manifold. The natural
map $\harm^{m}_{abs} \to H^{m}(M)$ is an isomorphism.
\end{thm}

Dually $\harm^m_E$ is isomorphic to the kernel of the homomorphism
$\inc : H^m_{cpt}(M) \to H^m(M)$ induced by the natural chain inclusion
$\Omega^*_{cpt}(M) \to \Omega^*(M)$. 
If $M$ has a single end (i.e. the cross-section $X$ is connected) then the
long exact sequence for relative cohomology of $(M,X)$ shows that
$\inc : H^{1}_{cpt}(M) \to H^{1}(M)$ is injective. Hence

\begin{cor}
\label{degonehodgecor}
Let $M^{n}$ be an asymptotically cylindrical manifold which has a single end
(i.e. the cross-section $X$ is connected). Then $\harm^{1}_{E} = 0$, and
$\harm^{1}_{0} \to H^{1}(M)$ is an isomorphism.
\end{cor}

\section{Ricci-flat deformations of $G$-metrics}
\label{ricdefgsec}

\subsection{Deformations of $G$-metrics}
\label{gdefsub}

Let $G$ be one of the Ricci-flat holonomy groups $SU(n)$, $Sp(n)$, $Spin(7)$
or $G_{2}$, and $M$ a compact \gmfd.
We explained in \S \ref{prelimsec} how a \gmetric{} on a manifold $M$
of the appropriate dimension can be defined in terms of a \gstr, i.e.
a section of a subbundle $\gspace \subset \Lambda^{*}T^{*}M$, which is
torsion-free and in particular closed.
In order to prove theorem \cptopenthm{} we will use that deformations
of \gstr s are unobstructed, and the existence of \emph{pre-moduli spaces}.

The tangent space to $\Gamma(\gspace)$ at a \gstr{} $\chi$
consists of the sections of the bundle of point-wise tangents to $\gspace$
at $\chi$, which is a vector bundle $E_{\chi} \subseteq \Lambda^{*}T^{*}M$
associated to the \gstr.
$E_{\chi}$ is a bundle of forms, so the Hodge Laplacian acts
on~$\Gamma(E_{\chi})$. When $\chi$ is torsion-free this is the same as
the Lichnerowicz Laplacian from \S \ref{harmholsub}.

The group $\mathcal{D}$ of diffeo\-morphisms of $M$ isotopic to the
identity acts on the space of torsion-free \gstr s by pull-backs
and the quotient is the moduli space $\defstrg$ of torsion-free \gstr s.
Goto \cite{goto04} proves that the deformations of torsion-free \gstr s
are unobstructed in the following sense:

\begin{prop}
\label{gpremodprop}
Let $G = SU(n)$, $Sp(n)$, $Spin(7)$ or $G_{2}$, $M$ a compact \gmfd,
and $\chi$ a torsion-free \gstr{} on $M$.
Then there is a submanifold $\qcalg$ of the space of $C^{1}$ \gstr s such that
\begin{enumerate}
\item \label{smoothitem}
the elements of $\qcalg$ are smooth torsion-free \gstr s,
\item the tangent space to $\qcalg$ at $\chi$ is the space of harmonic
sections of $E_{\chi}$,
\item the natural map $\qcalg \to \defstrg$ is a homeomorphism onto a
neighbourhood of $\chi\mathcal{D}$ in~$\defstrg$.
\end{enumerate}
\end{prop}

The spaces $\qcalg$ are pre-moduli spaces of torsion-free \gstr s and
can be used as coordinate charts for~$\defstrg$, which is thus a smooth
manifold.
The pre-moduli space $\qcalg$ near $\chi$ can be chosen to be invariant under
the stabiliser $\chi$. In fact

\begin{prop}
\label{autinvprop}
Let $\chi \in \mathcal{X}$, and let $\cali_{\chi} \subseteq \cald$ be the
stabiliser of $\chi$.
If $\qcalg$ is $\cali_{\chi}$-invariant and small enough then
$\cali_{x}$ acts trivially on $\qcalg$
and $\mathcal{I}_{\chi'} = \mathcal{I}_{x}$ for all $\chi' \in \qcalg$.
\end{prop}

\begin{proof}
Because the tangent space to $\qcalg$ consists of harmonic forms, a
neighbourhood of $\chi$ can be immersed in (a direct sum of copies of) the
de Rham cohomology of $M$. Because elements of $\cali_{\chi}$ act trivially on
cohomology they must fix such a neighbourhood.
The reverse inclusion $\cali_{\chi'} \subseteq \cali_{\chi}$ follows from
\mbox{\cite[Theorem $7.1(2)$]{ebin70}}.
\end{proof}

\subsection{Killing vector fields}
\label{killsub}

Before we discuss the deformations of Ricci-flat metrics we make some remarks
about \emph{Killing vector fields}. These are the infinitesimal isometries of a
Riemannian manifold $(M,g)$,
i.e. vector fields $V$ such that the Lie derivative
$\mathcal{L}_{V}g$ vanishes.

\begin{defn}
Given a metric $g$ on $M$ let
$\delta^{*} : \Omega^{1}(M) \to \Gamma(S^{2}(T^{*}M))$ be the symmetric part of
the Levi-Civita connection
$\nabla : \Omega^{1}(M) \to \Gamma(T^{*}M \otimes T^{*}M)$.
\end{defn}

\noindent
The formal adjoint $\delta$ of $\delta^{*}$ is the restriction of
$\nabla^{*} : \Gamma(T^{*}M \otimes T^{*}M) \to \Omega^{1}(M)$ to
the symmetric part $\Gamma(S^{2}(T^{*}M))$.

\begin{prop}[{\cite[Lemma $1.60$]{besse87}}]
\label{delstarprop}
Let $g$ be a Riemannian metric on a manifold $M$ and $V$ a vector field.
Then $\mathcal{L}_{V}g = 2\delta^{*}V^{\flat}$, where $V^{\flat}$ denotes the
$1$-form $g(V, \cdot)$.
\end{prop}

The second Bianchi identity implies that
\begin{equation}
\label{bianchieq}
(2\delta + d\:\tr)Ric = 0
\end{equation}
for any Riemannian metric. The operator $2\delta + d\:\tr$ is sometimes called
the Bianchi operator, and it also satisfies the following useful identity.

\begin{lem}[{\cite[Equation (14)]{kovalev05}}]
\label{delstarlem}
If $(M,g)$ is a Ricci-flat manifold then
\[ (2\delta + d\:\tr)\delta^{*} = \triangle . \]
\end{lem}

\begin{proof}
The anti-symmetric part of $\nabla$ on $\Omega^{1}(M)$ is $\half d$, so
$\delta^{*} = \nabla - \half d$.
Also $\tr \; \delta^{*} = d^{*}$ on~$\Omega^{1}(M)$.
Using the Weitzenböck formula $\triangle = \nabla^{*}\nabla - Ric$ we obtain
\[ (2\delta + d\:\tr)\delta^{*} =
2\nabla^{*}\nabla - \nabla^{*}d +d\:\tr\:\delta^{*} =
2\nabla^{*}\nabla - d^{*}d - dd^{*} = \triangle . \qedhere\]
\end{proof}

\begin{prop}
\label{killingprop}
Let $(M,g)$ be a Ricci-flat manifold. If $V$ is a Killing field then
the $1$-form $V^{\flat}$ is harmonic.
If $M$ is compact then the converse also holds.
\end{prop}

\begin{proof}
$\delta^{*}V^{\flat} = 0 \Rightarrow \triangle V^{\flat} = 0$ by
lemma \ref{delstarlem}.
Trivially $\nabla V^{\flat} = 0 \Rightarrow \delta^{*}V^{\flat} = 0$,
and if $M$ is compact then
$\triangle V^{\flat} = \nabla^{*}\nabla V^{\flat} = 0
\Rightarrow \nabla V^{\flat} = 0$
by integration by parts.
\end{proof}

This implies that, for any of the Ricci-flat holonomy groups $G$,
the space of infinitesimal automorphisms of a compact
\gmfd{} is $(\harm^{1})^{\sharp}$.

\subsection{Deformations of Ricci-flat metrics}
\label{ricdefsub}

We summarise some deformation theory for Ricci-flat metrics. This is
essentially taken from the explanation of the deformation theory for Einstein
metrics in \cite[\S 12C]{besse87} (in turn based on Koiso \cite{koiso83}),
specialised to the Ricci-flat case.
The main difference in this presentation is a slightly simplified technique
in the `slice argument'.

Let $M^{n}$ be a compact manifold. The diffeomorphism group $\mathcal{D}$ acts
on the space of Ricci-flat metrics on $M$ by pull-backs.
We define the moduli space $\defmetz$ of Ricci-flat metrics to be the quotient
of the space of Ricci-flat metrics by $\mathcal{D}$. (We do not divide by the
rescaling action of $\bbrp$ too, as is done in \cite{besse87}.)

Take $k \geq 2$, and let $g$ be a Ricci-flat Riemannian metric on $M$.
In order to study a neighbourhood of $g\mathcal{D}$ in $\defmetz$
we use the usual technique of considering a transverse slice for the
diffeomorphism action. Such a slice argument is explained very carefully in
Ebin~\cite{ebin70}. In the current setting it is, however, possible to use
elliptic regularity to avoid some of the technical subtleties of Ebin's
argument. As in \cite[\S 6.7]{jn1}, where a similar simplification is used, one
advantage compared with Ebin's approach is that it is easier to extend to the
asymptotically cylindrical case.

We include the space of smooth Riemannian metrics in the Hölder space
$\holda{k}(S^{2}T^{*}M)$, and let $\kpcal{D}$ be the $\holda{k+1}$ completion
of $\mathcal{D}$ ($\kpcal{D}$ is generated by $\exp$ of $\holda{k+1}$
vector fields).
By proposition \ref{delstarprop} the tangent space to the $\kpcal{D}$-orbit
at $g$ is $\delta^{*}_{g}\holda{k+1}(\Lambda^{1})$. 
Let $\slt$ be the kernel of $2\delta_{g} + d\:\tr_{g}$ in
$\holda{k}(S^{2}T^{*}M)$.
Because $g$ is Ricci-flat, harmonic $1$-forms are parallel and therefore
\ltwoorth{} to the image of $2\delta_{g} + d\:\tr_{g}$. It follows from
lemma \ref{delstarlem} and the
Fredholm alternative for $\triangle_{g}$ on $\Omega^{1}(M)$ that
there is a direct sum decomposition
\begin{equation*}
\holda{k}(S^{2}T^{*}M) = \delta^{*}_{g}\holda{k+1}(\Lambda^{1}) \oplus \slt .
\end{equation*}
We use a neighbourhood $\mathcal{S}$ of $g$ in $\slt$ as a slice for the
$\mathcal{D}$-action.

\begin{rmk}
This is not exactly the same choice of slice as in \cite{besse87}.
It has been used before by Biquard \cite{biquard00}
and Kovalev \cite{kovalev05}.
\end{rmk}

Let $\qcalz$ be the space of Ricci-flat (not \textit{a priori} smooth)
metrics in $\mathcal{S}$ -- this
is the \emph{pre-moduli space of Ricci-flat metrics} near $g$.
The linearisation of the Ricci curvature functional at a Ricci-flat metric
is given by (cf. \cite[Equation $(12.28')$]{besse87})
\begin{equation}
\label{infriceq}
(DRic)_{g}h = \triangle_{L} h + \delta^{*}_{g}(2\delta_{g} + d\:\tr_{g})h ,
\end{equation} 
where $\triangle_{L}$ denotes the Lichnerowicz Laplacian on $S^{2}T^{*}M$
in the sense of definition \ref{lichdef}.
In particular, on the tangent space $\slt$ to the slice
the linearisation reduces to $\triangle_{L}$.
This is elliptic so its kernel has finite dimension. Moreover, the kernel
of $\triangle_{L}$ is contained in~$K$:
differentiating the Bianchi identity (\ref{bianchieq})
at the Ricci-flat metric $g$ gives
\[ (2\delta_{g} + d\:\tr_{g})(DRic)_{g} = 0 , \]
and hence
\[ \triangle_{L}h = 0 \; \Rightarrow \;
\triangle(2\delta_{g} + d\:\tr_{g})h = 0
\; \Rightarrow \; (2\delta_{g} + d\:\tr_{g})h = 0 . \]

\begin{defn}
The space of \emph{infinitesimal Ricci-flat deformations} of $g$ is the
kernel $\varepsilon(g)$ of $\triangle_{L}$ in
$\Gamma(S^{2}(T^{*}M))$.
\end{defn}

If $h \in \Gamma(S^{2}T^{*}M)$ is tangent to a curve of Ricci-flat metrics
in the slice $\mathcal{S}$ then of course $h \in \varepsilon(g)$.
The converse is not true; in general there may be elements in $\varepsilon(g)$
which are not tangent to any curve of Ricci-flat metrics.
Thus $\qcalz$ need not be a manifold with tangent space $\varepsilon(g)$.

The image of $DRic_{g}$ is the \ltwoorth{} complement $K'$ to $\varepsilon(g)$
in $K$. Let $P_{g}$ be the \ltwoorth{} projection to $K'$.
The Ricci curvature functional is real analytic. We can apply the implicit
function theorem to the composition
\begin{equation}
\label{ricprojeq}
F : \cals \to K' : \;\: h \mapsto  P_{g}Ric(h)
\end{equation}
to deduce that there is a real analytic submanifold $Z \subseteq \mathcal{S}$
whose tangent space at $g$ is precisely $\varepsilon(g)$ and which contains
$\qcalz$ as a real analytic subset. The
analyticity implies that \emph{if} every element of $\varepsilon(g)$ is tangent
to a curve of Ricci-flat metrics then in fact $\qcalz$
contains a neighbourhood of $g$ in~$Z$. Thus the pre-moduli space
$\qcalz$ is a manifold in this case (cf. \cite[Corollary $3.5$]{koiso83}).

Note that since $K$ is invariant under the isometry group $\cali_{g}$ of $g$ we
may take $\cals$, $Z$ and $\calq$ to be invariant too. An analogue of
proposition \ref{autinvprop} holds.

\begin{prop}
\label{isominvprop}
For any $g' \in \calq$ sufficiently close to $g$,
$\cali_{g'} \subseteq \cali_{g}$. Moreover, the identity components of
$\cali_{g'}$ and $\cali_{g}$ are equal.
\end{prop}

\begin{proof}
The inclusion $\cali_{g'} \subseteq \cali_{g}$ follows from
\cite[Theorem $7.1(2)$]{ebin70}.
Proposition \ref{killingprop} implies that the dimension of the isometry groups
of Ricci-flat metrics is $b^{1}(M)$, so if $\cali_{g'} \subseteq \cali_{g}$
then the identity components must be equal.
\end{proof}

The elements of $Z$ are smooth by elliptic regularity (since the linear part of
the equation $F(h) = 0$ defining $Z$ is $\triangle_{L}h = 0$), and
when $\qcalz = Z$ it is relatively straight-forward to deduce from the
submersion theorem that $\qcalz \to \defmetz$ is open.
In general one needs to do a little bit of extra work.

\begin{thm}
\label{koisothm}
Let $M$ be a compact manifold and $g$ a Ricci-flat metric on $M$.
Let $\qcalz$ be the pre-moduli space of Ricci-flat metrics near $g$,
and $\cali_{g}$ the stabiliser of $g$ in $\cald$.
Then $\qcalz/\cali_{g}$ is homeomorphic to
a neighbourhood of $g\cald$ in $\defmetz$.
In particular, if every element of $\varepsilon(g)$ is integrable then
$\defmetz$ is an orbifold near $g\cald$.
\end{thm}

\begin{proof}
We wish to extend (\ref{ricprojeq}) to a function on a neighbourhood
$U$ of $g$ in $\holda{k}(S^{2}T^{*}M)$ such that $F^{-1}(0)$ is a manifold
containing the Ricci-flat metrics in $U$ \emph{and} ensure that
$Z\kpcal{D} \cap U \subseteq F^{-1}(0)$.
Then we apply the submersion theorem to deduce that $Z$ contains
representatives for all diffeomorphism classes in $F^{-1}(0)$ close to $g$.

By the inverse function theorem, any element of a small neighbourhood $U$ of
$g$ can be written as $k + \phi^{*}g'$, with $k \in K'$, $\phi \in \kpcal{D}$
and $g' \in Z$. Using proposition \ref{isominvprop},
\begin{equation*}
f : U \to g\kpcal{D}, \;\: k + \phi^{*}g' \mapsto \phi^{*}g
\end{equation*}
is a well-defined smooth function.
If $f(h) = \phi^{*}g$ then $P_{f(h)}$ is a projection to $\phi^{*}K'$,
and we can take
\begin{equation}
F : U \to K', \;\: h \mapsto P_{g}P_{f(h)}Ric(h) .
\end{equation}
Then $DF_{g}$ maps $K'$ onto itself, so $F^{-1}(0)$ is a submanifold of $U$
by the implicit function theorem. By construction it contains both the
Ricci-flat metrics in $U$ and $Z\kpcal{D} \cap U$. Now
\begin{equation}
\label{zsubeq}
Z \times \kpcal{D} \to F^{-1}(0)
\end{equation}
is an open map near $(g, id)$ by the submersion theorem (it is smooth because
elements of $Z$ are). This implies that
any smooth Ricci-flat metric $g'$ near $g$ is $\kpcal{D}$-equivalent to an
element of $Z$, which must in fact lie in $\qcalz$ because Ricci-flatness is a
diffeomorphism-invariant property. Since isometries between smooth Riemannian
metrics are smooth (see Myers and Steenrod \cite[Theorem~8]{myers39}),
$g'$ is in fact $\cald$-equivalent to an element of $\qcalz$. 
In other words, $\qcalz \to \defmetz$ is open.

Proposition \ref{isominvprop} implies that in fact $\qcalz \to \defmetz$ is
injective up to the action of the stabiliser $\cali_{g}$ and, since $\cali_{g}$
is compact, that the action on $\qcalz$ factors through a finite group
(cf. \cite[$12.25$]{besse87}).
\end{proof}

\begin{rmk}
\label{orbirmk}
Clearly the argument would give the same result even if we were to consider the
moduli space of Ricci-flat metrics given by dividing by the action of
the full diffeomorphism group of $M$.
\end{rmk}

\begin{rmk}
In \cite[Lemma $2.6$]{koiso83} Koiso uses instead of $\mathcal{S}$ a slice
constructed by Ebin \cite{ebin70}, and shows that any
Einstein metric in this slice is smooth.
\end{rmk}

\subsection{Proof of theorem \cptopenthm}
\label{proof1sub}

Let $G$ be one of the Ricci-flat holonomy groups,
$M$ a compact \gmfd, $\Gamma(\gspace)$ the space of \gstr s on $M$ and
\begin{equation}
\label{mbundleeq}
m : \Gamma(\gspace) \to \Gamma(S^{2}T^{*}M), \;\: \chi \mapsto g_{\chi}
\end{equation}
the natural map that sends a \gstr{} to the metric it defines.
In order to prove theorem \cptopenthm{} we show first that for any
torsion-free
\gstr{} $\chi$ the derivative of $m$ maps the tangent space to the
pre-moduli space $\qcalg$ at $\chi$ onto the space $\varepsilon(g_{\chi})$
of infinitesimal Ricci-flat deformations.

The tangent space to $\Gamma(\gspace)$ at $\chi$ is the space of differential
forms $\Gamma(E_{\chi})$, where $E_{\chi} \subseteq \Lambda^{*}T^{*}M$
is a vector subbundle associated to the \gstr{} defined by $\chi$.
Fibre-wise $\gspace$ is a $GL(\bbr^{n})$-orbit and $E_{\chi}$ is the tangent
space $\glalg_{n}\chi$ to the orbit. Because $m$ is $GL(\bbr^{n})$-equivariant
its derivative takes $a\chi \mapsto ag_{\chi}$ for any $a \in \glalg_{n}$,
which maps onto the fibre of $S^{2}T^{*}M$.
Hence the derivative
\begin{equation}
\label{mdereq}
Dm_{\chi} : \Gamma(E_{\chi}) \to \Gamma(S^{2}T^{*}M)
\end{equation}
is surjective. Furthermore, the derivative
is $G$-equivariant with respect to the \gstr{} defined by $\chi$.
Since $\triangle_{L}$ is the Lichnerowicz Laplacian
on $S^{2}T^{*}M$,
lemma \ref{lichlem} implies that the diagram below commutes.

\begin{diagram}[PostScript=Rikicki,balance]
\Gamma(E_{\chi})   &    \rTo^{Dm_{\chi}}  &  \Gamma(S^{2}T^{*}M) \\
\dTo^\triangle\;\; &  & \dTo_{\triangle_{L}} \\
\Gamma(E_{\chi})   &    \rTo^{Dm_{\chi}}  &  \Gamma(S^{2}T^{*}M) 
\end{diagram}

\noindent Hence

\begin{lem}
\label{mdersurjlem}
If $\chi$ is a torsion-free \gstr{} then $Dm_{\chi}$ maps the harmonic sections
of $E_{\chi}$ onto the space $\varepsilon(g_{\chi})$ of infinitesimal
Ricci-flat deformations.
\end{lem}

So let $\chi$ be any torsion-free \gstr{} on $M$ and $\qcalg$ the pre-moduli
space of torsion-free \gstr s near $\chi$. As described in subsection
\ref{ricdefsub}, there is a slice at $g_{\chi}$ for the $\mathcal{D}$-action on
the metrics, the Ricci-flat metrics in the slice are a real analytic subset
of a submanifold $Z$, and the tangent space
to $Z$ at $\chi$ is $\varepsilon(g_{\chi})$.
Let $P : F^{-1}(0) \to Z$ be the composition of
a smooth local right inverse to the submersion (\ref{zsubeq}) with the
projection to the first factor. $F^{-1}(0)$ contains the Ricci-flat metrics
near $g_{\chi}$, and $P$ can be viewed as a local projection to the slice:
$P(g')$ is $\cald$-equivalent to $g'$ for any Ricci-flat $g'$ close
to~$g_{\chi}$. Then
\begin{equation}
\label{pomeq}
P \circ m : \qcalg \to Z
\end{equation}
is a well-defined smooth map and lemma \ref{mdersurjlem}
means that its derivative at $\chi$ is surjective.
Therefore every element of $\varepsilon(g_{\chi})$ is tangent to
a path of Ricci-flat metrics, so $\qcalz$ is a manifold.
By the submersion theorem, $\defmetg$ (the image of $\defstrg$ in $\defmetz$)
contains a neighbourhood of~$g\cald$.

The pre-images of $g_{\chi}$ under $m$ are defined by differential forms
which are harmonic with respect to $g_{\chi}$. By Hodge theory they represent
distinct cohomology classes.
Let $\cali_{g_{\chi}} \subseteq \cald$ be the isometries of $g_{\chi}$ isotopic
to the identity.
Because $\cali_{g_{\chi}}$ acts trivially on cohomology it must fix
the fibre over of $m$ over $g_{\chi}$, so $\cali_{g_{\chi}} = \cali_{\chi}$.
Now, if $g' \in \qcalz$ then
$g' = \phi^{*}m(\chi')$ for some $\chi' \in \qcalg$ and $\phi \in \kpcal{D}$
because (\ref{pomeq}) is a submersion. As $\cali_{\chi}$
acts trivially on $\qcalg$ by proposition \ref{autinvprop} it follows
that the conjugate $\cali_{g_{\chi}}^{\phi}$ fixes $g'$.
But then
$\cali_{g_{\chi}}^{\phi} \subseteq \cali_{g'} \subseteq \cali_{g_{\chi}}$
by proposition \ref{isominvprop}, so in fact
$\cali_{g_{\chi}}^{\phi} = \cali_{g_{\chi}}$.
Hence $\cali_{g_{\chi}}$ fixes any $g' \in \calq$.

Now theorem \ref{koisothm} implies that $\qcalz$ is homeomorphic to a
neighbourhood of $\defmetz$. Thus $\defmetg$ is a manifold near $g\cald$
and the proof of theorem \cptopenthm{} is complete.

\subsection{The asymptotically cylindrical case}
\label{eacopensub}

The proof of theorem \cptopenthm{} only used the compactness assumption
to access certain deformation results for \gstr s and Ricci-flat metrics.
For the cases $G = G_{2}$ and $Spin(7)$ there are pre-moduli spaces of EAC
\gstr s, with properties analogous to proposition \ref{gpremodprop}.

If $M$ is an EAC \gmfd{}, let $\defstrp$ denote the quotient of the space of
torsion-free EAC \gstr s on $M$ by the group $\cald$ of EAC diffeomorphisms
of $M$ isotopic to the identity.

\begin{prop}
\label{cylgpremodprop}
Let $G = Spin(7)$ or $G_{2}$, $M$ an EAC \gmfd{}
and $\chi$ a torsion-free EAC \gstr{} on $M$.
Then there is a submanifold $\qcalg$ of the space of $C^{1}$ \gstr s such that
\begin{enumerate}
\item
the elements of $\qcalg$ are smooth EAC torsion-free \gstr s,
\item the tangent space to $\qcalg$ at $\chi$ is the space of bounded harmonic
sections of $E_{\chi}$,
\item the natural map $\qcalg \to \defstrp$ is a homeomorphism onto a
neighbourhood of $\chi\pcal{D}$ in~$\defstrp$.
\end{enumerate}
\end{prop}

\begin{proof}
See \cite[\S 6]{jn1} for the $G_{2}$ case, and \cite[\S 4.3]{jnthesis}
for the $Spin(7)$ case.
\end{proof}

In order to prove the theorem \eacopenthm{}, the EAC version of theorem
\cptopenthm{}, it therefore suffices to explain how to set up the deformation
theory for EAC Ricci-flat metrics.
Below we define the slices with same equations as in the compact case in
\S \ref{ricdefsub} and use the same reasoning as for deformations of EAC
\gtmfd s in \mbox{\cite[\S 6.7]{jn1}} to make the slice arguments work on
EAC manifolds.
The resulting approach is similar to that of Kovalev \cite{kovalev05},
who considers Ricci-flat deformations of EAC Calabi-Yau manifolds.

Let $M^{n}$ be a manifold with cylindrical ends and cross-section $X^{n-1}$.
Let $\defmetp$ be the quotient of the space of EAC Ricci-flat metrics (with
any exponential rate) by the group $\pcal{D}$ of EAC diffeomorphisms of $M$
isotopic to the identity.
We pick an EAC Ricci-flat metric $g$ on $M$ and study a neighbourhood
of $g\pcal{D}$ in $\defmetp$.
By definition, the asymptotic limit of $g$ is a cylindrical metric
$dt^{2} + g_{X}$ on $X \times \bbr$,
where $g_{X}$ is a Ricci-flat metric on~$X$.

We work with weighted Hölder spaces of sections. Let $k \geq 2$,
$\alpha \in (0,1)$, and $\delta > 0$ be less than the exponential rate of $g$.
The metric $g$ defines a Hodge Laplacian on $1$-forms and a
Lichnerowicz Laplacian on symmetric bilinear forms, which are both
asymptotically translation-invariant operators.
We require that $\delta$ is small enough that the Laplacians are Fredholm
on $\holdad{k}$ spaces, as we may according to theorem \ref{lapindthm}.

We proved in \S \ref{ricdefsub} that there is a real analytic submanifold
$Z \subset \holda{k}(S^{2}T^{*}X)$ which contains representatives of all
diffeomorphism classes of Ricci-flat metrics on $X$ close to~$g_{X}$.
Its tangent space $T_{g_{X}}Z = \varepsilon(g_{X})$ is the space of
Lichnerowicz harmonic sections of $S^{2}T^{*}X$.

Let $\scrmkz$ denote the space of $\holda{k}$ metrics on $M$ which are
$\holdad{k}$-asymptotic to cylindrical metrics $dt + g_{X}^{2}$ such that
$g_{X} \in Z$.
If $\rho$ is a cut-off function for the cylinder then $\rho Z$ can be identified
with a space of bilinear forms on $M$, and $\scrmkz$ is an open subset
\begin{equation*}
\scrmkz \subset \holdad{k}(S^{*}T^{*}M) + \rho Z .
\end{equation*}
Similarly let $\caldkpz$ be the set of EAC diffeomorphisms with rate $\delta$
which are asymptotic to elements of the isometry group $\cali_{g_{X}}$ of
$g_{X}$.
Then $\scrmkz$ contains representatives of all diffeomorphism classes of
Ricci-flat metrics near $g$ and, because $Z$ is $\cali_{g_{X}}$-invariant,
proposition \ref{isominvprop} implies that any isometry between elements of
$\scrmkz$ must lie in $\caldkpz$ (a similar argument for simplifying the
problem by a slice at the boundary was used to study the moduli space
of torsion-free EAC \gtstr s in \cite[Lemma 6.24]{jn1}).
We therefore identify a slice in $\scrmkz$ for the action of $\caldkpz$ at $g$.
The tangent space to $\scrmkz$ at $g$ is
\begin{equation*}
T_{g}\scrmkz = \holdad{k}(S^{*}T^{*}M) \oplus \rho \varepsilon(g_{X}) .
\end{equation*}
The tangent space at the identity of $\caldkpz$ corresponds to vector fields
which are $\holdad{k}$-asymptotic to translation-invariant Killing vector
fields on the cylinder, i.e. to elements of~$(\harm^{1}_{\infty})^{\sharp}$,
where $\harm^{1}_{\infty}$ denotes the translation-invariant harmonic
$1$-forms on the cylinder $X \times \bbr$.
By proposition \ref{delstarprop} the tangent space to the $\caldkpz$-orbit
at $g$ is
\begin{equation*}
\delta^{*}_{g}(\holdad{k}(\Lambda^{1}) \oplus \rho \harm^{1}_{\infty}) .
\end{equation*}
Let $\slt$ be the kernel of $2\delta_{g} + d\:\tr_{g}$ in $T_{g}\scrmkz$.

\begin{lem}
Let $M$ be a Ricci-flat EAC manifold with a single end.
Then
\begin{equation}
T_{g}\scrmkz = \slt \oplus
\delta^{*}_{g}(\holdad{k}(\Lambda^{1}) \oplus \rho \harm^{1}_{\infty}) .
\end{equation}
\end{lem}

\begin{proof}
$(2\delta_{g} + d\:\tr_{g})\delta^{*} = \triangle_{g}$ according to
lemma \ref{delstarlem}, so it suffices to show that the image of
$2\delta_{g} + d\:\tr_{g} : T_{g}\scrmkz \to \holdad{k-1}(\Lambda^{1})$
is contained in the image of
\begin{equation*}
\triangle : \holdad{k+1}(\Lambda^{1}) \oplus \rho \harm^{1}_{\infty}
\to \holdad{k-1}(\Lambda^{1}) .
\end{equation*}
It follows from theorem \ref{lapindthm} that this has index $0$,
so its image is the \ltwoorth{} complement to its kernel $\harm^{1}_{0}$,
the space of bounded harmonic $1$-forms.

Now, if $h \in T_{g}\scrmkz$ and $\beta \in \harm^{1}_{0}$ then
the difference between $\inner{\delta_{g}h, \beta}$ and
$\inner{h, \delta^{*}_{g}\beta} = 0$ is the divergence of the contraction
of $h$ with $\beta$. The boundary condition on $h$ ensures that the
asymptotic limit of the contraction has no $dt$-component, so the integral
of the divergence is $0$. Hence
\[ \linner{\delta_{g}h, \beta} = 0 . \]
The hypothesis that $M$ has a single end ensures that the asymptotic limit
of $\beta$ has no $dt$-component (corollary \ref{degonehodgecor}),
so integration by parts also applies to show that
\mbox{$\linner{d\:\tr_{g}h, \beta} = 0$}.
Thus the image of $2\delta_{g} + d\:\tr_{g}$ is \ltwoorth{} to $\harm^{1}_{0}$.
\end{proof}

Now we can use a real analytic $\cali_{g}$-invariant submanifold
$\cals \subset \scrmkz$ with $T_{g}\cals = \slt$ as a slice for the
$\caldkpz$-action.
Let $\qcalz \subset \cals$ be the subset of Ricci-flat metrics.
As in the compact case $\qcalz$ is an analytic subset of an analytic
submanifold $Z' \subset \cals$, defined as the zero set of the composition
of the Ricci functional $\cals \to \holdad{k-2}(S^{2}T^{*}M)$ with the
projection onto the image of its derivative at $g$.
On $K$ the derivative of the Ricci functional is the Lichnerowicz Laplacian,
so $T_{g}Z'$ is the space of harmonic sections of $S^{2}T^{*}M$,
exponentially asymptotic to sections of $S^{2}T^{*}X$ (i.e. the asymptotic
limit has no $dt$-components). This is the space of infinitesimal
Ricci-flat EAC deformations $\varepsilon(g)$.

In general we can use regularity and arguments like in
\cite[\S 6.6]{jn1} to show that $Z$ consists of smooth
EAC metrics, and then extend the proof of theorem \ref{koisothm}, thus proving
that $\qcalz/\cali_{g}$ is homeomorphic to a neighbourhood
of $g$ in $\defmetp$.

We have now set up the deformation theory for EAC Ricci-flat metrics
that is required, together with the unobstructedness of deformations of
torsion-free EAC \gstr s for $G = G_{2}$ and $Spin(7)$
stated in proposition \ref{cylgpremodprop},
in order to prove theorem \eacopenthm{} by the same argument as for
the compact case.

\bibliographystyle{plain}
\bibliography{g2geom}

\end{document}